\documentclass{amsart}

\usepackage{amsfonts}
\usepackage{amssymb}
\usepackage{amsmath,amsxtra,amssymb}

\newtheorem{theorem}{Theorem}[section]
\newtheorem{lemma}[theorem]{Lemma}
\newtheorem{proposition}[theorem]{Proposition}
\newtheorem{corollary}[theorem]{Corollary}
\theoremstyle{definition}

\theoremstyle{remark}
\newtheorem{remark}[theorem]{Remark}

\newcommand{\A}{{\mathcal{A}}}

\newcommand{\field}[1]{\mathbb{#1}}
\newcommand{\C}{{\field{C}}}

\newcommand{\la}{{\langle}}
\newcommand{\ra}{\rightarrow}
\newcommand{\id}{{\iota}}
\newcommand{\ot}{{\otimes}}
\newcommand{\tp}{{\widehat{\otimes}}}
\newcommand{\vtp}{{\overline{\otimes}}}

\newcommand{\T}{{\mathcal{T}}}
\newcommand{\B}{{\mathcal{B}}}

\newcommand{\om}{{\omega}}
\newcommand{\fee}{{\varphi}}

\newcommand{\G}{{\field{G}}}
\newcommand{\LL}{{L^\infty(\G)}}
\newcommand{\LO}{{L^1(\G)}}
\newcommand{\LT}{{L^2(\G)}}
\newcommand{\Tr}{{\mathcal{T}(L^2(\G))}}
\newcommand{\TG}{{\T_\star(\G)}}

\newcommand{\LLL}{{L^\infty(\hat\G)}}
\newcommand{\LLLL}{{L^\infty(\hat\G')}}

\newcommand{\loneqg}{L^1(\mathbb{G})}
\newcommand{\linfqg}{L^\infty(\mathbb{G})}

\numberwithin{equation}{section}

\begin{document}

\title[Duality, Cohomology, and Geometry of Quantum Groups]{Duality, Cohomology, and Geometry \\ of Locally Compact Quantum Groups}

\author{Mehrdad Kalantar}

\address{Mehrdad Kalantar \newline School of Mathematics and Statistics,
Carleton University, Ottawa, Ontario, K1S 5B6, Canada}
\email{mkalanta@math.carleton.ca}

\author{Matthias Neufang}

\address{Matthias Neufang \newline School of Mathematics and Statistics,
Carleton University, Ottawa, Ontario, K1S 5B6, Canada}
\email{mneufang@math.carleton.ca}

\vspace{-0.7cm}
\address{${}$ \vspace{-0.4cm}\newline Universit\'{e} Lille 1 -- Sciences et Technologies,
UFR de Math\'{e}matiques, Laboratoire de Math\'{e}matiques Paul Painlev\'{e}  (UMR CNRS 8524),
59655 Villeneuve d'Ascq C\'{e}dex, France}
\email{Matthias.Neufang@math.univ-lille1.fr}
${}$ \\[-5ex]
\address{${}$ \vspace{-0.4cm}\newline The Fields Institute for Research in
Mathematical Sciences, Toronto, Ontario,
Canada M5T 3J1}
\email{mneufang@fields.utoronto.ca}

%    Abstract is required.
\begin{abstract}
In this paper we study  various
convolution-type algebras associated with a locally compact quantum group
from cohomological and geometrical points of view. The quantum group duality
endows the space of trace class operators
over a locally compact quantum group with two products which are operator versions of
convolution and pointwise multiplication, respectively; we
investigate the relation between these two products, and derive a formula
linking them.
Furthermore, we define some canonical module structures on these
convolution algebras, and prove that certain topological properties of a
quantum group, can be completely characterized in terms of cohomological properties of these modules.
We also prove a quantum group version of a theorem of  Hulanicki characterizing group amenability.
Finally, we study the Radon--Nikodym property of the $L^1$-algebra of locally compact quantum groups.
In particular, we obtain a criterion that distinguishes discreteness from the Radon--Nikodym property
in this setting.
\end{abstract}

\maketitle

\section{Introduction}
The most fundamental objects in abstract harmonic analysis are algebras of functions on
a locally compact group $G$, endowed with the 
convolution, respectively, pointwise product,
such as the group algebra $L^1(G)$ and
the Fourier algebra $A(G)$.
Despite being dual to each other in a canonical way, these two products cannot be compared and linked
to one another in an obvious way, because they live on very different spaces.
However, as we shall show in this paper, it is the duality of locally compact quantum groups $\G$
that provides a common ground on which these two products can be studied
simultaneously on one space, namely the trace class operators $\Tr$.

Our goal in this paper is to study locally compact quantum groups $\G$ from cohomological and
geometrical points of view. The fact that the co-multiplication of a locally compact quantum group 
is implemented by its fundamental unitary, enables one to lift the product
of $\LO$ to $\Tr$.
Therefore, $\Tr$ can be canonically endowed with two products which
arise from $\G$ and $\hat\G$; in the classical case of
a locally compact group $G$, these products are indeed operator
versions of the convolution and the pointwise products.

The paper is organized as follows. The preliminary definitions and results which
are needed, are briefly recalled in section 2.
In section 3, we first define the quantum version of the convolution
and pointwise products on the space $\Tr$ of trace class operators
on the Hilbert space $\LT$ of a locally compact quantum group $\G$.
We then study the basic properties of these algebras, and use
the duality theory of locally compact quantum groups to derive a formula linking
the two products associated with $\G$ and $\hat \G$.

In section 4,  we consider various module structures associated with
convolution algebras over a locally compact quantum group, and investigate their cohomological properties. 
We show that topological properties of a locally compact quantum group $\G$
are equivalent to cohomological properties of certain convolution algebras over $\G$.
In \cite{M-ths}, the second-named author  introduced and studied the above-mentioned convolution product
on $\mathcal{T}(L^2(G))$ for a locally compact group $G$. The corresponding results on
the equivalence of topological and cohomological properties in this situation were obtained
in \cite{Pir}.
We also establish in this section a quantum group version of a theorem of Hulanicki
stating that a discrete group is amenable if and only if its left regular representation is an
isometry on positive elements of $l^1(G)$: indeed, we show that for any co-amenable
locally compact quantum group $\G$, the latter condition is equivalent to
co-amenability of the dual $\hat \G$, i.e., to $\G$ having Reiter's property $(P_2)$,
as introduced and studied in \cite{VolMatt}.

In the last section, for a locally compact quantum group $\G$, we study a geometric property of $\LO$, namely the Radon--Nikodym property (RNP).
While, for a locally compact group $G$, the space $L^1(G)$ has the RNP if and only if $G$ is discrete,
the dual statement, with $L^1(G)$ replaced by the Fourier algebra $A(G)$, is not true in general.
So, the RNP and discreteness are not equivalent for arbitrary locally compact quantum groups.
We characterize the difference  between both properties in this general setting in terms of a covariance condition.

\par
The results in this paper are based on \cite{mehrdad-thesis},
written under the supervision of the second-named author.

\section{Preliminaries}

We recall from  \cite {KV2} and \cite {Vaes} that a (von
Neumann algebraic) \emph{locally compact quantum group} $\G$ is a
quadruple $(\linfqg, \Gamma, \varphi, \psi)$, where $\linfqg$ is a
von Neumann algebra with a co-multiplication
$$\Gamma: \linfqg\to \linfqg \vtp \linfqg,$$, 
and $\varphi$ and  $\psi $ are  (normal faithful semifinite) left and right
Haar weights on $\linfqg$, respectively. For each
locally compact quantum group $\G$, there exist a  \emph{left fundamental unitary
operator}  $W$ on $L^{2}(\G, \varphi)\otimes L^{2}(\G, \varphi)$
and a \emph{right fundamental unitary operator} $V$ on $L^{2}(\G,
\psi)\otimes L^{2}(\G, \psi)$
which satisfy  the  \emph{pentagonal relation}
 \begin{equation}
\label {F.pentagonal} W_{12} W_{13}W_{23} = W_{23} W_{12}
\quad\mbox{and} \quad V_{12} V_{13}V_{23} = V_{23} V_{12}.
\end{equation}
The co-multiplication $\Gamma$ on $\linfqg$ can be expressed as
\begin{equation}
\label {F.com} \Gamma(x) = W^{*}(1\otimes x)W= V(x\otimes 1)V^{*}
\quad(x \in \linfqg).
\end{equation}
We can identify $L^{2}(\G, \varphi)$ and $L^{2}(\G, \psi)$ (cf. \cite [Proposition 2.11]{KV2}), and we simply
use $L^{2}(\G)$ for this Hilbert space in the rest of this paper.

Let $\loneqg$ be the predual of $\linfqg$. Then the pre-adjoint of
$\Gamma$ induces on $\loneqg$ an associative completely
contractive multiplication
\begin{equation}
\label {F.mul} \star  :  f_{1} \otimes f_{2} \in L^{1}(\G)\hat
\otimes L^{1}(\G) \to f_{1} \star f_{2} = (f_{1} \otimes
f_{2})\circ \Gamma \in L^{1}(\G).
\end{equation}

A locally compact quantum group $\G$ is called \emph{co-amenable} if
$\loneqg$ has a bounded left (equivalently, right or two-sided)
approximate identity (cf. \cite[Theorem 3.1]{B-T}).

The \emph{left regular representation} $\lambda : L^{1}(\G) \to
\B(L^{2} (\G))$ is defined by
 \[
\lambda :  L^{1}(\G)\ni f   \mapsto \lambda(f) = (f\otimes \iota)(W)
\in \B(L^{2}(\G)),
 \]
which is an injective and completely contractive algebra homomorphism
from
 $L^{1}(\G)$ into $\B(L^{2} (\G))$. Then  
 $$L^{\infty} (\hat \G)={\{\lambda(f): f\in \loneqg\}}''$$
is the von Neumann algebra associated with the dual quantum group
$\hat \G$. Analogously, we have the  \emph{right regular
representation} $\rho : L^{1}(\G) \to \B(L^{2} (\G))$ defined by
 \[
\rho : L^{1}(\G)\ni f   \mapsto \rho(f) = (\iota\otimes f)(V) \in
\B(L^{2}(\G)),
 \]
which is also an injective and completely contractive algebra
homomorphism from
 $L^{1}(\G)$ into $\B(L^{2} (\G))$. Then 
 $$L^{\infty}(\hat \G') ={\{\rho(f): f\in \loneqg\}}''$$
  is the von Neumann algebra associated with  the dual (commutant) quantum group $\hat \G'$. It
follows that 
$$W \in \linfqg \bar \otimes L^{\infty}(\hat \G) \ \ \ \text{and} \ \ \ V \in L^{\infty}(\hat \G')\bar \otimes \linfqg.$$ 
 
We obtain the corresponding reduced quantum group
$C^*$-subalgebra
\[
C_{0}(\G) = \overline{\{(\iota \otimes \hat f)(W) : \hat f\in L^{1}(\hat \G)\}}^{\|\cdot\|} =
\overline{\{\hat f' \otimes \iota)(V) : \hat f'\in L^{1}(\hat \G')}^{\|\cdot\|}
\]
of $\linfqg$ with the co-multiplication
\[
\Gamma :  C_{0}(\G)\to  M(C_{0}(\G)\otimes
C_{0}(\G)),
\]
where $M(C_{0}(\G)\otimes C_{0}(\G))$ is the multiplier algebra of
the minimal $C^*$-algebra tensor product $C_{0}(\G)\otimes
C_{0}(\G)$.

Let $M(\G)$ denote the operator dual $C_{0}(\G)^{*}$ of $C_{0}(\G)$.
The space $M(\G)$ is a completely
contractive \emph{dual Banach algebra} (i.e., the multiplication
on $M(\G)$ is separately weak$^*$ continuous),
and $M(\G)$ contains $L^{1}(\G)$ as a norm closed
two-sided ideal via the embedding 
$$L^1(\G) \rightarrow M(\G) \ : \ f \mapsto f_{|C_0(\G)}.$$

If $G$ is a locally compact group, then $C_{0}(\G_{a})$ is the
$C^*$-algebra $C_{0}(G)$ of continuous functions on $G$ vanishing at
infinity, and $M(\G_{a})$ is the measure algebra $M(G)$ of $G$.
Correspondingly, $C_{0}(\hat \G_{a})$ is the left group $C^*$-algebra
$C^{*}_{\lambda}(G)$ of $G$, and $C_{0}(\hat
\G^{\prime}_{a})$ is the right group $C^*$-algebra $C^{*}_{\rho}(G)$
of $G$. Hence, we have $M(\hat \G_{a})= B_{\lambda}(G)$ and
$M(\hat \G_{a}^{\prime})= B_{\rho}(G)$.

%-----------------------------
%  PROPOSITION 2.2
%-----------------------------

%%% Remark 2.3

We also briefly recall some standard definitions and notations from the cohomology theory of Banach algebras (cf. \cite{Hel}).
Actually, as one might expect, here in the general setting of locally compact quantum groups, we need to
take the quantum (operator space) structure of the underlying Banach spaces into
account as well. So we work in the category of operator spaces;
we shall define our module structures, and their corresponding objects, in the quantized
Banach space category as well.

A completely bounded linear map $\sigma: X\ra Y$ from an operator space $X$ into an operator space $Y$ is called
{\it admissible} if it has a completely bounded right inverse.

Let $\A$ be a Banach algebra and $P$ be a right $\A$-module. $P$ is called \emph{projective} if for all
$\A$-modules $X$ and $Y$, any admissible morphism $\sigma: X\ra Y$, and any morphism
$\rho: P\ra Y$, there exists a morphism $\phi: P\ra X$ such that $\sigma\circ\phi = \rho$.

Denote by $X\triangleleft\A$ the closed linear span of the set
$$\{x\triangleleft a :a\in \A , x\in X\}\subseteq X.$$
 Then $X$ is called {\it essential} if $X\triangleleft\A = X$.

Many categorical statements which hold in the category of Banach spaces,
also hold in this setting with an obvious slight categorical modification.
In particular, the following result which is well-known in the classical setting (cf. \cite{Hel2});
we will use it frequently in our work.
\begin{theorem}\label{33383}
an essential right $\A$-module $X$ is projective 
if and only if there exists a morphism $\psi:X\rightarrow X\tp\A$ such that $\mathfrak{m}\circ\psi = \id_X$, 
where $\mathfrak{m}:X\tp \A\rightarrow X$ is the canonical module action morphism, and $X\tp\A$ is regarded as a right $\A$-module, 
via the action $(x\otimes a)\triangleleft b = x\otimes ab$.
\end{theorem}

The case of left modules and bi-modules are analogous.

\section{Convolution and Pointwise Products for Locally Compact Quantum Groups}
In this section we define a quantum analogous of the convolution and pointwise
products for a locally compact quantum group, study the basic properties, and
state a formula linking them.

Let $\G$ be a locally compact quantum group, and $V \in \LLLL\vtp \LL$ its right
fundamental unitary. We can lift the co-products $\Gamma$ and $\hat\Gamma$ to
$\B(\LT)$, still using the same notation, as follows:
\[\begin{array}{lll}
\Gamma : \B(L^2(\G)) \rightarrow \B(L^2 (\G))\vtp \B(L^2 (\G)), && x \mapsto V (x\otimes 1) V^*;\\[1ex]
\hat\Gamma : \B(L^2(\G)) \rightarrow  \B(L^2(\G))\vtp \B(L^2 (\G)), && x \mapsto \hat V' (x\otimes 1)\hat V'^*.
\end{array}
\]
Then the preadjoint maps
$${\Gamma}_{\ast}, {\hat{\Gamma}}_{*} : \Tr\tp \Tr\rightarrow \Tr$$
define two different completely contractive products on the space of trace class operators $\Tr$. 
We denote these products by $\star$
and $\bullet$ respectively. We also denote by $\T_\star(\G)$ and $\T_\bullet(\G)$ 
the (quantized) Banach algebras $(\Tr,\star)$ and $(\Tr,\bullet)$, respectively.
If $\G=L^\infty(G)$ for a locally compact group $G$, then $\T_\star(\G)$ is
the convolution algebra introduced by Neufang in \cite{M-ths}.

Applied to the classical setting, i.e., the commutative
and co-commutative cases, the following lemma justifies
why the above products are considered as quantum versions
of convolution and point-wise products.

\begin{proposition} \label{211} The canonical quotient map $\pi:\T_{\star}(\G)\twoheadrightarrow L^1(\G)$
and the trace map $tr: \T_{\star}(\G)\rightarrow \C$ are
Banach algebra homomorphisms.
\end{proposition}
\begin{proof}
First part follows from the fact that $\Gamma(\LL)\subseteq\LL\vtp\LL$, and the 
second part is an easy consequence of the identity $\Gamma(1) = 1\ot 1$.
\end{proof}
The above Proposition allows us to define (right) $\T_\star(\G)$-module
structures on $L^1(\G)$ and $\C$ as follows:
\[\label{859476}
f\triangleleft\rho = f\star\pi(\rho) \ \ \  \text{and} 
\ \ \ \lambda\triangleleft\rho = \lambda tr(\rho),
\]
where $\rho\in\T_\star(\G)$, $f\in\LO$, and $\lambda\in\C$.
We will show later that some of the topological
properties of $\G$ can be deduced from these module structures.

But, first we prove some properties 
of the lifted co-products and their induced products.
\begin{proposition}\label{212}
Let $x\in \B(\LT)$. If $\Gamma (x)$ = $y\otimes 1$ for some $y\in \B(\LT)$ then
we have $x = y\in \LLL$.
\end{proposition}
\begin{proof} We have
\begin{equation}\label{1039}
((\iota\otimes\omega)V)x = (\iota\otimes\omega)(V(x\otimes 1)) =
(\iota\otimes\omega)((y\otimes 1)V) = y((\iota\otimes\omega)V)
\end{equation}
for all $\omega\in \Tr$. Since 
$$\LLLL = \overline{\{(\iota\otimes \omega)V: \omega\in\Tr\}}^{w^*},$$
it follows from (\ref{1039}) that 
$${\hat a}'x = y{\hat a}'$$
 for all ${\hat a}'\in \LLLL$. 
In particular for ${\hat a}' = 1$, it follows that $x = y$,
and since $\hat a' x = x\hat a'$ for all $\hat a'\in\LLLL$, we have $x\in\LLL$.
\end{proof}
\begin{lemma}\label{213}
Let $x\in \B(\LT)$. If $\Gamma (x) = 1\otimes y$, for some $y\in \B(\LT)$, then $x\in \LLLL$ and $y\in \LL$.
\end{lemma}
\begin{proof} 
Since $(1\otimes y) = V(x\otimes 1)V^*\in \B(\LT)\vtp\LL$, we have $y\in \LL$.
Similarly, $x\otimes 1 = V^*(1\otimes y)V\in\LLLL\vtp\LL$ implies that $x\in\LLLL$.
\end{proof}
\begin{corollary}\label{214}
Let $x\in \B(\LT)$. If $\Gamma (x) = 1\otimes x$, then $x\in \C 1$.
\end{corollary}
\begin{proof}  If $\Gamma (x)$ = $1\otimes x$, the Lemma \ref{213} implies that
$x\in \LL \cap \LLLL$ which equals $\C 1$.
\end{proof}

Now we investigate the relation between these two products
on $\Tr$, and find a formula (\ref{formmula}) linking them.
\begin{proposition}\label{224}
For $\rho$, $\xi$ and $\eta$ in $\Tr$, the following two relations hold:
\begin{eqnarray*}
\rho \star (\xi\bullet \eta)  &=& \la\eta,1\rangle\rho \star\xi;\\
\rho \bullet (\xi \star \eta) &=& \la\eta,1\rangle\rho \bullet\xi.
\end{eqnarray*}
\end{proposition}
\begin{proof} Let $x\in \LL$ and $\hat x\in \LLL$. Then we have
\begin{eqnarray*}
\la\rho \star (\xi\bullet \eta),x\hat x\rangle &=&
 \langle\rho \otimes (\xi\bullet \eta),\Gamma(x\hat x)\rangle
\\ &=& \la\rho \otimes \xi\otimes \eta,(\iota\otimes {\hat\Gamma})(\Gamma (x\hat x))\rangle   \\
&=& \la\rho \otimes \xi\otimes \eta,(\Gamma(x)\otimes 1)(\hat x\otimes 1\otimes 1)\rangle\\ 
&=& \la\eta,1\rangle\la\rho \otimes \xi,\Gamma(x)(\hat x\otimes 1)\rangle\\
&=& \la\eta,1\rangle\la\rho \star\xi,x\hat x\rangle,
\end{eqnarray*}
which, by weak$^*$ density of the span of the
set $\{x\hat x : x\in\LL , \hat x \in\LLL\}$ in $\B(\LT)$,  implies the first formula.
The second relation follows along similar lines.
\end{proof}

Since there are two different multiplications on $\Tr$ arising from 
$\G$ and $\hat\G$, it is tempting to
consider the corresponding two actions at the same time 
by defining a bi-module structure on $\Tr$, using these two products. 
But one can deduce from the above proposition, that multiplication from the left and right via these products,
is not associative, and so we cannot turn $\Tr$ 
into a $\T_\star(\G)-\T_\bullet(\G)$ bimodule in this fashion. However,
next theorem will provide us with a way of doing so.
\begin{theorem}\label{223}
For $\rho$, $\xi$ and $\eta$ in $\Tr$, the following relation holds:
\begin{equation}\label{formmula}
(\rho \star \xi)\bullet \eta \ = \ (\rho\bullet\eta)\star\xi.
\end{equation}
Equivalently denoting by $\mathfrak{m}$ and $\hat{\mathfrak{m}}$
the product maps corresponding to $\star$ and $\bullet$, respectively,
we have
\[
\mathfrak{m}\circ(\hat{\mathfrak{m}}\ot\id) = 
\hat{\mathfrak{m}}\circ(\mathfrak{m}\ot\id)\circ(\id\ot\sigma)
\]
on the triple (operator space) projective tensor product of $\Tr$ with itself; here $\sigma$ is the flip map.
\end{theorem}
\begin{remark}
This theorem shows that the dual products on quantum groups ``anti-commute'':
the minus sign of a usual anti-commutation relation in an algebra (with respect to a given product)
is replaced by the flip map when comparing two different products.
\end{remark}
\begin{proof} Let $x\in \LL$ and $\hat x\in \LLL$. Then we have:
\begin{eqnarray*}
\la(\rho \star \xi)\bullet \eta,x\hat x\rangle  &=&
\la(\rho \star \xi)\otimes \eta,{\hat\Gamma}(x\hat x)\rangle
 \\ %&=&  \la(\rho \star \xi)\otimes \eta,(x\otimes 1){\hat\Gamma}(\hat x)\rangle\\
& =&  \la\rho \otimes \xi\otimes \eta,(\Gamma \otimes \iota)[(x\otimes 1){\hat\Gamma}(\hat x)]\rangle
 \\ &=&  \la\rho \otimes \xi\otimes \eta,(\Gamma (x)\otimes 1){{\hat\Gamma}(\hat x)}_{13}\rangle\\
 &=&  \la\rho \otimes \eta\otimes \xi,{\Gamma (x)}_{13}({\hat\Gamma}(\hat x)\otimes 1)\rangle
 \\ &=&  \la\rho \otimes \eta\otimes \xi,({\hat\Gamma} \otimes \iota)[\Gamma(x)(\hat x\otimes 1)]\rangle\\
& =&  \la(\rho \bullet \eta)\otimes \xi,\Gamma(x\hat x)\rangle
 \\ &=&  \la(\rho \bullet \eta)\star \xi,x\hat x\rangle.
\end{eqnarray*}
Hence, theorem follows, again  by weak$^*$ density of the span of the
set $$\{x\hat x : x\in\LL , \hat x \in\LLL\}\subseteq \B(\LT).$$
\end{proof}
Theorem \ref{223} has even more significance:
the quantum group duality may be encoded  by this relation. 
In fact, one might be able to start from this relation on trace class operators
 on a Hilbert space, with some extra
conditions, to arrive to an equivalent axiomatic definition for locally compact quantum groups.
We intend to address this project in a subsequent paper.
\begin{proposition} \label{225} The space $\Tr$ becomes a $\T_\star(\G)^{op}-\T_\bullet(\G)$ 
bimodule via the actions
$$\eta\triangleright\rho = \rho\star\eta \  \ \  \ \text{and} 
\ \ \   \ \rho\triangleleft\xi = \rho\bullet\xi,$$
where $\rho \in \Tr$, $\eta \in \T_\star (\G)$ and $\xi\in \T_\bullet(\G)$.
\end{proposition}
\begin{proof} We only need to check the associativity of
the left-right action. For this, using Theorem \ref{223}, we obtain
\[
(\eta\triangleright\rho)\triangleleft\xi = 
(\rho\star\eta)\bullet\xi 
= (\rho\bullet\xi)\star\eta
= \eta\triangleright(\rho\triangleleft\xi).
\]
\end{proof}

The following proposition is known and has been stated in many different places.
\begin{proposition}\label{215} 
Let $\G$ be a locally compact quantum group. Then the following hold:
 \begin{itemize}
  \item[(1)] $\LO$ has a left (right) identity if and only if $\G$ is discrete.
\item[(2)] $\LO$ has a bounded left (right) approximate identity if and only if $\G$ is co-amenable.
 \end{itemize}
\end{proposition}
In contrast to the last proposition, we have the following.
\begin{proposition}\label{216} Let $\G$ be a locally compact quantum group. Then the following hold:
\begin{itemize}
  \item[(1)]   $\T_\star(\G)$ does not have a left identity, unless $\G$ is trivial,
  and it has a right identity if and only if $\G$ is discrete;
  \item[(2)]  $\T_\star(\G)$ does not have a left approximate identity, unless $\G$ is trivial,
  and it has a bounded right approximate identity if and only if $\G$ is co-amenable.
\end{itemize}
\end{proposition}
\begin{proof} \ (1): let $\omega_0\in\Tr$ be a
non-zero normal functional whose restriction to $\LL$ is zero. Since we have
$$\Gamma(\B(\LT))\subseteq\B(\LT)\vtp\LL,$$
it follows that
\[
\langle\rho\star\omega_0,x\rangle = \langle\omega_0,(\rho\ot\id)\Gamma(x)\rangle = 0,
\]
for all $\rho\in\Tr$ and $x\in\B(\LT)$, which obviously implies
that there does not exist a left identity, unless $\G$ is trivial (equal to $\C$).

Now, let $\G$ be discrete, $e \in L^1(\G)$ be the
unit element, and $\tilde e\in\Tr$ be a norm preserving weak$^\ast$-extension of $e$. 
Then we have
\begin{eqnarray*}
\langle\rho\star\tilde e,x\hat{x}\rangle &=& \langle\rho\otimes\tilde e,\Gamma(x)
(\hat{x}\otimes 1)\rangle \\ &=& 
\langle\hat{x}\rho \otimes \tilde e,\Gamma(x)\rangle 
= \langle\pi(\hat{x}\rho)\star e,x\rangle 
\\ &=&  \langle\pi(\hat{x} \rho),x\rangle
= \langle\hat{x}\rho,x\rangle
= \langle\rho,x\hat{x}\rangle
\end{eqnarray*}
for all $\rho \in \T_\star(\G)$, $x \in \LL$ and $\hat{x} \in\LLL$, 
where $\pi:\Tr\twoheadrightarrow \LO$ is the canonical quotient map.
Since the span of the set 
$\{ x\hat x  :  x\in L^\infty(\G) , \hat x\in L^\infty(\hat\G) \}$ is weak$^*$ dense in 
$\B(L^2(\G))$, it follows that $\tilde e$ is a right identity for $\T_\star(\LT)$.

Conversely, assume that $\T_{\star}(\G)$ has a right identity $\tilde e$. 
Then, since by Proposition \ref{211}, the map 
$\pi:\T_{\star}(\G)\rightarrow L^1(\G)$ is a surjective homomorphism, 
$\pi(\tilde e)$ is clearly a right identity for $L^1(\G)$, 
whence $\G$ is discrete by Proposition \ref{215}.\\
(2): Similarly to the first part, one can show that $\T_{\star}(\G)$ cannot possess a left approximate identity,
unless it is trivial (equal to $\C$).

Let $\G$ be co-amenable. Then, by \cite[Theorem 3.1]{B-T},
there exists a net $(\xi_i)$ of unit vectors in $\LT$ such that
$$\| V^*(\eta\ot\xi_i) - \eta\ot\xi_i\|\rightarrow 0$$
 for all 
unit vectors $\eta\in\LT$.
Now, for all $x\in\B(\LT)$ and $\eta\in \LT$ with $\|x\| = \|\eta\| = 1$ we have
\begin{eqnarray*}
\left |\la \omega_\eta\star\omega_{\xi_i} - \omega_\eta , x \rangle\right | &=& 
|\la V(x\ot 1)V^*(\eta\ot\xi_i), \eta\ot\xi_i\rangle - \la (x\ot 1)(\eta\ot\xi_i), \eta\ot\xi_i\rangle| \\ 
&=& |\la (x\ot 1)(V^*(\eta\ot\xi_i) - \eta\ot\xi_i) , V^*(\eta\ot\xi_i)\rangle +\\ 
& & \la (x\ot 1) (\eta\ot\xi_i) , V^*(\eta\ot\xi_i) - \eta\ot\xi_i\rangle|\\
&\leq& 2 \| V^*(\eta\ot\xi_i) - \eta\ot\xi_i\|\rightarrow 0.
\end{eqnarray*}
Since the span of the set $\{ \omega_\eta : \eta\in \LT \}$
is norm dense in $\Tr$, it follows that
$(\omega_{\xi_i})$ is a right bounded approximate identity for $\T_\star(\G)$.

Conversely, if $\T_{\star}(\G)$ has a right bounded approximate identity, then a similar
 argument to the proof of part (1) shows that $\G$ is co-amenable.
\end{proof}

\begin{proposition} \label{217} A locally compact quantum group 
$\G$ is compact if and only if there exists a state $\tilde\varphi\in \T_\star(\G)$ such that
\begin{equation}\label{e217}
\langle\rho\star\tilde\varphi,x \hat{x}\rangle = \langle\rho\star\tilde\varphi,\hat{x}x\rangle
= \langle\rho,\hat{x}\rangle\langle\tilde\varphi,x\rangle
\end{equation}
for all $x \in \LL$, $\hat{x} \in \LLL$ and $\rho\in \T_\star(\G)$.
\end{proposition}

\begin{proof}
Suppose that $\G$ is compact with normal Haar state $\varphi$, and 
$\tilde\varphi \in \T_\star(\G)$ is a norm preserving extension of $\varphi$.
Then $\tilde\fee$ is a state (since $\|\tilde\fee\| = \tilde\fee(1)=1$), and we have
\begin{eqnarray*}
 \la\rho\star\tilde\varphi,x \hat{x}\rangle & = & \la \rho\otimes\tilde\varphi,\Gamma(x)(\hat{x}\otimes 1)\rangle = 
\la\hat{x}\rho\otimes\tilde\varphi,\Gamma(x)\rangle \\
& = & \la\pi(\hat{x}\rho)\otimes\pi(\tilde\varphi),\Gamma(x)\rangle  =  \la\pi(\hat{x}\rho)\star\varphi,x\rangle \\
& = & \la\pi(\hat{x}\rho),1\rangle\la\varphi,x\rangle  =  \la\hat{x}\rho,1\rangle\la\varphi,x\rangle\\
& = & \la\rho,\hat{x}\rangle\la\varphi,x\rangle  =  \la\rho,\hat{x}\rangle\la\tilde\varphi,x\rangle.
\end{eqnarray*}
In a similar way, we can show that $\la\rho\star\tilde\varphi,\hat{x}x\rangle = \la\rho,\hat{x}\rangle\la\tilde\varphi,x\rangle$.

Conversely, suppose such a state $\tilde\fee\in \T_\star(\G)$ exists. Let $\fee = \pi(\tilde\fee)\in L^1(\G)^+$ and
$f\in L^1(\G)$, and let $\tilde f\in \T_\star(\G)$ be a weak$^\ast$-extension of $f$.
Then, by putting $\hat x = 1$ in equation (\ref{e217}), we have
\[
 \la f\star\varphi, x\rangle
= \la\tilde f \star\tilde\varphi,x\rangle  =  \la\tilde f,1\rangle\la\tilde\varphi,x\rangle
 = \la f,1\rangle\la\varphi,x\rangle
\]
for all $x\in \LL$.
Hence, $\fee$ is a left invariant state in $L^1(\G)$, and so $\G$ is compact, by \cite[Proposition 3.1]{B-T}.
\end{proof}

\section{Cohomological Properties of Convolution Algebras}
The following result was proved in the more general setting of Hopf-von Neumann algebras
in \cite[Theorem 2.3]{Aristov2}.
\begin{proposition}\label{222}
$\C$ is a projective $\T_{\star} (\G)$-module if and only if $\G$ is compact.
\end{proposition}
In the following, we want to  prove a statement similar to Proposition \ref{222}, for
discreteness of $\G$.
But the situation is  more subtle in this case. There are some technical difficulties
which arise when one tries to link the quantum group structure to the quantum
Banach space structure. This happens mainly because the latter
is essentially defined based on the Banach space structure of these algebras,
and do not seem to see all aspects of the quantum group structure.
These technical issues appear also in some of the open problems in this theory,
and seem to be a major subtle point (c.f. \cite{Daws}).

To avoid such difficulties, in the rest of this section,
we assume that the morphisms are completely contractive, rather than just completely bounded.
\begin{theorem}\label{3310}
Let $\G$ be a locally compact quantum group. Then the following are equivalent:
\begin{itemize}
\item[(1)]
there exists a normal conditional expectation $E:\B(\LT)\ra \LL$ which satisfies $\Gamma\circ E = (E\otimes E)\Gamma$;
\item[(2)]
there exists a normal conditional expectation $E:\B(\LT)\ra \LL$ which satisfies $\Gamma\circ E = (E\otimes \iota)\Gamma$;
\item[(3)]
there exists a normal conditional expectation $E:\B(\LT)\ra \LL$ which satisfies $E(\LLL)\subseteq \C1$;
\item[(4)]
$\G$ is discrete.
\end{itemize}
\end{theorem}
\begin{proof} $(1)\Leftrightarrow (2):$
This follows from the facts that $E=\iota$ on $\LL$, and that $\Gamma(\B(\LT))\subseteq \B(\LT)\vtp \LL$.\\
$(1)\Rightarrow (3):$ Let $\hat x\in\LLL$. We have:
\begin{eqnarray*}
\Gamma(E(\hat x)) = (E\otimes E)\Gamma(\hat x)
= (E\otimes E)(\hat x\otimes 1) = E(\hat x)\otimes 1,
\end{eqnarray*}
which implies that $E(\hat x)\in\C1$.\\
$(3)\Rightarrow (4):$ Assumption (3) implies that
 $$E\in \mathcal{CB}^{\sigma, \LLL}_\LL(\B(\LT)),$$
and hence it follows by \cite[Theorem 4.5]{J-N-R} that 
there exists a right centralizer $\hat m\in C_{cb}^r(L^1(\hat\G))$ such that
$$E = \hat\Theta^r(\hat m).$$
Now, define a complex-valued map $\hat f$ on $\LLL$ such that
$$E(\hat x) = \hat f(\hat x)1$$
for all $\hat x\in \LLL$. Since $E$ is a unital
linear normal positive map, $\hat f$ is a normal state on $\LLL$,
and for every $\hat\omega\in L^1(\hat\G)$ and $\hat x\in\LLL$ we have
\begin{eqnarray*}
\langle \hat m(\hat\omega),\hat x\rangle = \langle \hat\omega,\hat\Theta^r(\hat m)(\hat x)\rangle
= \langle \hat\omega,E(\hat x)\rangle 
= \langle \hat\omega,\hat f(\hat x)1\rangle = \langle \hat\omega,1\rangle \hat f(\hat x).
\end{eqnarray*}
Hence $\hat m(\hat\omega) = \langle \hat\omega,1\rangle \hat f$.
Now, fix $\hat\omega_0\in L^1(\hat\G)$, then for all $\hat\omega\in L^1(\hat\G)$ we have
\begin{equation*}
\hat\omega\star\hat f  = \hat\omega\star\hat m(\hat\omega_0)  = \hat m(\hat\omega\star\hat\omega_0)
= \langle\hat\omega\star\hat\omega_0,1\rangle \hat f = \langle\hat\omega,1\rangle \hat f.
\end{equation*}
Hence, $\hat f$ is a normal left invariant state on $\LLL$, and therefore $\hat \G$ is compact by
\cite[Proposition 3.1]{B-T}, and (4) follows.\\
$(4)\Rightarrow (2):$ Let $e$ be the identity of $\LO$, and let $\tilde e\in\Tr$
be a norm-preserving extension of $e$. Define:
\[
E: \B(\LT)\ra \LL; \ \ \  x \mapsto (\tilde e\ot\id)\Gamma(x).
\]
Then $E$ is normal, unital and completely contractive, since both $(\tilde e\ot\id)$ and $\Gamma$ are,
which also implies that $\|E\| = 1$. For all $x\in \LL$ and $f\in\LO$ we have
\[
\langle f , E(x) \rangle = \langle f , (\tilde e\ot\id)\Gamma(x) \rangle =
\langle \tilde e\ot f , x \rangle = \langle f , x \rangle,
\]
which implies that $E^2 = E$, and $E$ is surjective. Hence, $E$ is a conditional expectation on $\LL$.
Now, for all $x\in\B(\LT)$, we have
\begin{eqnarray*}
\Gamma(E(x)) &=& \Gamma((\tilde e\ot\id)\Gamma(x)) 
= (\tilde e\ot\id\ot\id)((\id\ot\Gamma)\Gamma(x))\\ 
&=& (\tilde e\ot\id\ot\id)((\Gamma\ot\id)\Gamma(x))
= (((\tilde e\ot\id)\Gamma)\ot\id) \Gamma(x)
\\ &=& (E\ot\id)\Gamma(x).
\end{eqnarray*}
Hence, $\Gamma\circ E = (E\ot\id)\Gamma$, and $(2)$ follows.
\end{proof}
\begin{remark}\label{92min}
 One can easily modify the above proof to obtain a right version of Theorem \ref{3310}; then
in part (2) we have $\Gamma\circ E = (\id\ot E)\Gamma$, and in part (3), $E(\LLLL)\subseteq\C1$.
\end{remark}

\begin{corollary}\label{3311}
For a locally compact quantum group $\G$ the following are equivalent:
\begin{itemize}
\item[(1)]
there exists an isometric algebra homomorphism $\Phi:\LO\ra \T_\star(\G)$
such that $\pi\circ\Phi = \id_\LO$;
\item[(2)]
$\G$ is discrete.
\end{itemize}
\end{corollary}
\noindent
\begin{proof} If $\G$ is discrete, then $\Phi$ may be taken to be the pre-adjoint of the map
$E$ constructed in the proof of the implication $(4)\Rightarrow(2)$ in Theorem \ref{3310}.

For the converse, note that $\Phi^*:\B(\LT)\ra\LL$ is a normal surjective
norm-one projection, i.e., a normal conditional expectation. 
Moreover, for $x\in \B(\LT)$ and $\rho,\eta\in \T_\star(\G)$ we have
\begin{eqnarray*}
 \la \rho\otimes\eta , \Gamma (\Phi^*(x))\rangle &=& \la \rho\star\eta , \Phi^*(x)\rangle
\\ &=& \la \Phi(\rho\star\eta) , x\rangle 
\\ &=& \la \Phi(\rho)\star\Phi(\eta) , x\rangle 
\\ &=& \la \Phi(\rho)\otimes\Phi(\eta) , \Gamma(x)\rangle 
\\ &=& \la \rho\otimes\eta , (\Phi^*\otimes\Phi^*)\Gamma(x)\rangle,
\end{eqnarray*}
which implies $\Gamma \circ\Phi^* = (\Phi^*\otimes\Phi^*)\Gamma$,
and hence the theorem follows from Theorem \ref{3310}.
\end{proof}

As we promised earlier in this section, in the following (Theorem \ref{33393}), we prove that
discreteness of a locally compact quantum group $\G$,
can also be characterized in terms of projectivity of its convolution algebras.
We recall that here the morphisms are completely contractive maps, and
$\mathfrak{m}: \LO\tp\T_\star(\G)\ra\LO$ denotes the
canonical map associated with the module action. 
\begin{lemma}\label{90min}
If $\Phi:\LO\ra\T_\star(\G)$ is such that $\pi\circ\Phi = \id_\LO$, then
for all $\eta, \rho\in\T_\star(\G)$ we have:
\[
 \eta\star\Phi(\pi(\rho)) = \eta\star\rho.
\]
\end{lemma}
\begin{proof} Recall that $\Gamma(x)\in\B(\LT)\vtp\LL$ for all $x\in\B(\LT)$.
Therefore, we clearly obtain that $\eta\star\rho = \eta\star\pi(\rho)$. Hence, we have
\[
 \eta\star\Phi(\pi(\rho)) = \eta\star\pi(\Phi(\pi(\rho))) = \eta\star\pi(\rho) = \eta\star\rho.
\]
\end{proof}
\begin{lemma}\label{91min}
 Assume that $\Psi:\LO\tp\T_\star(\G)$ is a $\TG$-module morphism
which satisfies $\mathfrak{m}\circ\Psi = \id_\LO$. Then for any
$\hat x\in \LL$ we have $\Psi^*(1\ot\hat x)\in\C1$.
\end{lemma}
\begin{proof}
 Let $\hat x\in\LL$, then we have
\begin{eqnarray*}
 \langle f\star\pi(\rho) , \Psi^*(1\ot\hat x)\rangle &=&
 \langle \Psi(f\star\pi(\rho)) , 1\ot\hat x\rangle =
\langle \Psi(f\triangleleft\rho) , 1\ot\hat x\rangle \\ &=&
\langle \Psi(f)\triangleleft\rho , 1\ot\hat x\rangle =
\langle \Psi(f)\star\rho , 1\ot\hat x\rangle \\
&=& \langle \Psi(f)\ot\rho , 1\ot\Gamma(\hat x)\rangle  =
\langle \Psi(f)\ot\rho , 1\ot\hat x\ot 1\rangle \\
&=&\langle \Psi(f) , 1\ot\hat x\rangle \langle \rho , 1\rangle =
\langle f , \Psi^*(1\ot\hat x)\rangle \langle \rho , 1\rangle
\end{eqnarray*}
for all $f\in\LO$ and $\rho\in\TG$. Since $\pi:\TG\ra\LO$ is surjective,
we get
\begin{eqnarray*}
&& \langle f\ot g , \Gamma(\Psi^*(1\ot\hat x))\rangle = 
\langle f\star g , \Psi^*(1\ot\hat x)\rangle \\&=&
\langle f , \Psi^*(1\ot\hat x)\rangle \langle g , 1\rangle =
\langle f\ot g , \Psi^*(1\ot\hat x)\ot 1\rangle
\end{eqnarray*}
for all $f,g\in\LO$.
Hence we have $$\Gamma(\Psi^*(1\ot\hat x)) = \Psi^*(1\ot\hat x)\ot 1,$$
 which implies that $\Psi^*(1\ot\hat x) \in\C1$.
\end{proof}
\begin{theorem}\label{33393}
For a locally compact quantum group $\G$, the following are equivalent:
\begin{itemize}
\item[(1)]
$\LO$ is a projective $\T_\star(\G)$-module;
\item[(2)]
$\G$ is discrete.
\end{itemize}
\end{theorem}
\begin{proof}$(1)\Rightarrow(2):$ Assume that $\LO$ is a projective $\T_\star(\G)$-module. 
So, there exists a $\T_\star(\G)$-module morphism
$\Psi:\LO\ra \LO\tp\T_\star(\G)$ such that 
$$\mathfrak{m}\circ\Psi = \id_\LO.$$
Let $R$ be the unitary antipode of $\G$ (cf. \cite{K-V}). Then $R(x) = \hat J x\hat J$ for
all $x\in\LL$, where $\hat J$ is the modular conjugate associated with the dual Haar weight $\hat\fee$.
Using the same formula 
$$x\mapsto \hat J x\hat J \ \ \ \ \  \big(x\in\B(\LT)\big),$$
we can extend the map $R$ to $\B(\LT)$. 
Then it is clear that 
$$R(\LLLL) = \LLL.$$
Denote by $\chi$ the flip map $a\ot b\mapsto b\ot a$, and define the map
$$T:\B(\LT) \ra \LL, \ \ T:= \Psi^*\circ\chi\circ(R\ot R)\circ\Gamma.$$
We shall prove that the map $E:=T^2$ satisfies the conditions of (the right version of) part (3) of Theorem \ref{3310}
(see the Remark \ref{92min}).

First note that $T$ is normal and contractive.
Moreover, for $x\in \LL$, we have
\[
 T(x) = \Psi^*\circ\chi\circ(R\ot R)\circ\Gamma(x) = \Psi^*\circ\Gamma(R(x)) = R(x).
\]
This implies that
$$T^2_{|\LL}=\id,$$
and so $E^2 = E$. Hence, $E:\B(\LT)\ra\LL$ is a normal conditional expectation on $\LL$.
Now, for all $\hat x'\in\LLLL$ we have
\begin{eqnarray*}
 T(\hat x') &=& \Psi^*\circ\chi\circ(R\ot R)\circ\Gamma(\hat x')\\ &=&
\Psi^*\circ\chi\circ(R\ot R) (\hat x'\ot 1)\\ &=&
\Psi^*\circ\chi (R(\hat x')\ot 1) \\ &=&
\Psi^*(1\ot R(\hat x')).
\end{eqnarray*}
Hence, $E(\LLLL)\subseteq\C1$, by Lemma \ref{91min}, and so $\G$ is discrete by (the right version of) Theorem \ref{3310}.\\
$(2)\Rightarrow(1):$ Let $e\in\LO$ be the identity element,
and $\Phi:\LO\ra\T_\star(\G)$, as in Corollary \ref{3311}.
Define the map $\Psi:\LO\ra\LO\tp\T_\star(\G)$ by
\[
\Psi(f) = e\ot\Phi(f) \ \ \ \ (f\in\LO).
\]
Since $\pi\circ\Phi = \id_\LO$, we have $\mathfrak{m}\circ\Psi = \id_\LO$.
Moreover, using Lemma \ref{90min}, we have
\begin{eqnarray*}
\Psi(f\triangleleft\rho) &=& e\ot\Phi(f\triangleleft\rho) 
= e\ot \Phi(f\star\pi(\rho))\\
& =& e\ot\big(\Phi(f)\star\Phi(\pi(\rho))\big)
= e\ot\big(\Phi(f)\star\rho\big)\\
& =& \Psi(f)\triangleleft\rho
\end{eqnarray*}
for all $f\in\LO$ and $\rho\in\T_\star(\G)$. Therefore $\Psi$ is a morphism,
and so $\LO$ is projective.
\end{proof}

The next theorem was proved for the case of Kac algebras in  \cite[Theorem 6.6.1]{E-S},
but the proof in there is based on the
structure theory of discrete Kac algebras. Here we present a different
argument for the general case of locally compact quantum groups.
\begin{theorem} \label{228}
If $\G$ is both compact and discrete, then $\G$ is finite (dimensional).
\end{theorem}
\begin{proof} If $\G$ is compact, then $L^1(\G)$ is an ideal in
 $L^1(\G)^{**}$ with the left (equivalently, right) Arens product,
 by \cite[Theorem 3.8]{Run}. But since $\G$ is also
 discrete, $L^1(\G)$ is unital, and its unit is obviously also
 an identity element for the left Arens product of $L^1(\G)^{**}$.
Being a unital ideal (via the canonical embedding), $L^1(\G)$ must be equal to
$L^1(\G)^{**}$. So $L^1(\G)$ is reflexive, hence $\LL$ is, which
implies that $\LL$ is finite-dimensional, by \cite[Proposition 1.11.7]{L}.
\end{proof}

Using Theorem \ref{228}, we can now follow a similar idea as the proof of
\cite[Theorem 3.7]{Pir}, to prove a quantum version of the latter.
\begin{theorem}\label{124312}
Let $\G$ be a co-amenable locally compact quantum group. Then the following are equivalent:
\begin{itemize}
 \item[(1)]
 $\T_\star(\G)$ is biprojective;
\item[(2)]
 $\G$ is finite.
\end{itemize}
\end{theorem}
\begin{proof} $(1)\Rightarrow(2):$ Since $\G$ is co-amenable, Proposition \ref{216} implies that $\T_\star(\G)$ has a bounded right approximate identity.
Since $\LO$ and $\C$ are both essential $\T_\star(\G)$-modules, they are $\T_\star(\G)$-projective by 
 \cite[7.1.60]{Hel}, which implies that $\G$ is both compact and discrete, by
Proposition \ref{222} and Theorem \ref{33393}. Hence, $\G$ is finite by Theorem \ref{228}.\\
$(2)\Rightarrow(1):$ Consider the short exact sequence 
$$0\ra I \ra \T_\star(\G) \xrightarrow{\pi} \LO \ra 0,$$
 where $$I:= \{\rho\in\Tr: \rho_{|\LL} = 0\}.$$
Since $\G$ is finite, it is in particular a compact Kac algebra, 
and so $\LO$ is operator biprojective.
Hence, (1) follows from \cite[Lemma 4.2]{Pir}.
\end{proof}

We can also define a right $\LO$-module structure on $\Tr$, as follows:
\begin{equation}\label{mod}
 \rho\triangleleft f := (\rho\otimes f)\circ\Gamma \ \ \ \ (\rho\in\Tr, f\in\LO).
\end{equation}

\begin{theorem}
For a locally compact quantum group $\G$, the following are equivalent:
\begin{itemize}
\item[(1)]
there exists an isometric $\LO$-module map $\Phi:\LO\ra\Tr$ such that $\pi\circ\Phi=\id_\LO$;
\item[(2)]
$\G$ is discrete.
\end{itemize}
\end{theorem}
\begin{proof} If $\G$ is discrete, then the predual of the map $E$
constructed in the proof of the implication $(4)\Rightarrow(2)$ in Theorem \ref{3310},
is easily seen to satisfy the desired conditions.

Conversely, if such a map $\Phi$ exists, then it is straightforward to see that the map
$$E:=\Phi^*:\B(\LT)\ra \LL$$ enjoys the properties in part $(2)$ of
Theorem \ref{3310}, and so $\G$ is discrete.
\end{proof}

In the following, we shall consider another important cohomology-type
property for the convolution algebras associated with a locally compact quantum group $\G$, namely amenability.

Next theorem is in fact a generalization of a result due to Hulanicki
who considered the case $\G= L^\infty(G)$ for a discrete group $G$,
to the setting of locally compact quantum groups.
This result was proved in the Kac algebra case by Kraus and Ruan in
\cite[Theorem 7.6]{Kra-Ruan}.
But their argument is based essentially on the fact
that in the Kac algebra setting, 
the left regular representation is a $^*$-homomorphism, which does not hold
anymore in the general setting of locally compact quantum groups, so it appears that their proof
cannot be modified for the latter case.
Here we present a different argument, inspired by the proof of \cite[Theorem 2.4]{Pisi}.
\begin{theorem}\label{230}
Let $\G$ be a co-amenable locally compact quantum group. Then the following are equivalent:
\begin{itemize}
\item [(1)]
the left regular representation $\lambda:L^1(\G) \rightarrow L^\infty(\hat{\G})$
is isometric on $L^1(\G)^+$;
\item [(2)]
$\hat\G$ is co-amenable, i.e., $\G$ has Reiter's property $(P_2)$.
\end{itemize}
\end{theorem}
\begin{proof} $(1)\Rightarrow(2):$ 
We first show that (1) implies that 
$$\|1+\lambda(f)\| = 1 + \|f\|$$
for all $f\in\LO^+$. To show this, 
let $(e_\alpha)\in \LO^+$ be a bounded approximate identity, and $g\in\LO^+$ with $\|e_\alpha\| = \|g\| = 1$ for 
all $\alpha$. Then we have
\begin{eqnarray*}
 \|f\| + 1 &=& f(1) + 1 = f(1)g(1)+e_\alpha(1)g(1) \\ &=& \langle f\star g + e_\alpha\star g , 1 \rangle
 = \| f\star g + e_\alpha\star g\|\\
&\rightarrow & \| f\star g+g\|  = \| \lambda(f\star g) + \lambda(g)\|\\ 
&=& \| (\lambda(f) + 1)\lambda(g)\|
 \leq \|\lambda(f) + 1\|\, \|\lambda(g)\|\\ &=& \|\lambda(f) + 1\|
 \leq \|\lambda(f)\| + 1 = \|f\| + 1 ,
\end{eqnarray*}
which implies our claim. 
Since $\G$ is co-amenable, there exists $\varepsilon\in M(\G)^+_1$ such that 
$\lambda(\varepsilon) = 1$, by \cite[Theorem 3.1]{B-T}. Let
$$\mathcal{F}_{0} = \{F\cup\{\varepsilon\} : F\subseteq \LO^+_{1}, \ \text{ $F$ is finite}\}.$$
Then for each $F\in \mathcal{F}_{0}$ we have
$$\big\|\sum_{f\in F}f\big\| = \langle\sum_{f\in F} f  , 1\rangle = |F|.$$
So $\big\|\sum_{f\in F} \lambda(f)\big\| = |F|$, and therefore there exists a
sequence $(\xi_{n})$ of unit vectors in $L^2(\G)$ such that
$$\lim_n\big\|\sum_{f\in F}\lambda(f)\xi_{n}\big\| = |F|.$$
Now fix $f_{0}\in F$, and let $F' = F\backslash\{\varepsilon,f_{0}\}$.
Then we have 
$$\lim_n\big\|(\lambda(f_{0})\xi_{n} + \xi_{n}) + \sum_{f\in F'}\lambda(f)\xi_{n}\big\| = |F|,$$
but since
$$\big\|\sum_{f\in F'} \lambda(f)\xi_{n}\big\| \leq |\,F\,| - 2 \ \  \ \ \forall n\in\field{N},$$
it follows that $\lim_n\|\lambda(f_{0})\xi_{n} + \xi_{n}\|= 2$, which yields
$$\lim_n\|\lambda(f_{0})\xi_{n} - \xi_{n}\|_{2} = 0.$$
Since both $f_{0}\in F$ and $F\in F_{0}$ were arbitrary, there
exists a net $(\xi_{i})$ of unit vectors in $L^2(\G)$ such that
$$\|\lambda(f)\xi_{i} - \xi_{i}\| \rightarrow 0$$
for all $f \in\LO^+_{1}$,
and since $L^\infty(\G)$ is standard on $L^2(\G)$ we have
$$\|W(\eta\otimes\xi_{i}) - \eta\otimes\xi_{i}\| \ra 0$$
for all unit vectors $\eta\in L^2(\G)$. 
Hence, $\hat{\G}$ is co-amenable, by \cite[Theorem 3.1]{B-T}.\\
$(2)\Rightarrow(1):$ Let $\chi$ denote the flip map $a\ot b\mapsto b\ot a$.
Since $\hat{\G}$ is co-amenable and $\chi(W)$ is
an isometry, \cite[Theorem 3.1]{B-T} ensures the existence of a
net $(\xi_{i})$ of unit vectors in $L^2(\G)$ such that
$$\lim_i\|W(\eta\otimes \xi_i) - \eta\otimes \xi_i\|_2 = 0 , \ \ \ \ \forall \eta\in \LT.$$
Now, let $f\in L^1(\G)^+$. Since $L^\infty(\G)$ is in standard form in $\B(L^2(\G))$,
we have $f = \omega_{\zeta}$, for some $\zeta\in L^2 (\G)$ with $\|f\| =
\|\zeta\|$. Assuming $\|\lambda(f)\|\leq1$, we obtain:
\begin{eqnarray*}
1 &\geq& \lim_i |\langle\lambda(f)\xi_{i},\xi_{i}\rangle|
 = \lim_i| \langle(f\otimes \iota)W\xi_{i},\xi_{i}\rangle|\\
& = & \lim_i| \langle W(\zeta\otimes \xi_i),\zeta\otimes \xi_i\rangle| 
= \lim_i\|\zeta\otimes \xi_i\|^2\\
&=& \|\zeta\|^2 \ =\ \|f\|^2.
\end{eqnarray*}
So, $\|\lambda(f)\|\leq 1$ implies $\|f\| \leq 1$, therefore the conclusion follows.

The equivalence between co-amenability of $\hat \G$ and Reiter's property $(P_2)$ of
$\G$ is the statement of \cite[Theorem 5.4]{VolMatt}.
\end{proof}
\begin{remark}\label{231}
Note that the assumption of co-amenability of $\G$ is not necessary for 
the implication $(2)\Rightarrow(1)$.
Also, this condition is not necessary if $\G$ is a Kac algebra,
as an easy modification of our argument shows.
\end{remark}

\section{The Radon--Nikodym Property for $\LO$}
In the last part of this paper we shall investigate a geometric property 
for the convolution algebra $\LO$ of a locally compact quantum group $\G$,
namely the Radon--Nikodym property (in short: RNP).

The following are some well-known results concerning the Radon--Nikodym Property of Banach spaces (cf. \cite{Diest2}).
\begin{proposition}\label{024} ${}$
 \begin{enumerate}\item
The RNP is inherited by closed subspaces, and is stable under isomorphisms.
\item
If $H$ is a Hilbert space then $\mathcal{T}(H)$ has the RNP.
\item
Let $G$ be a locally compact group. Then $L^1(G)$ has the RNP if and only if $G$ is discrete.
\end{enumerate}
\end{proposition}
\begin{proposition}\label{218} Let $\G$ be a locally compact quantum group. 
If there exists $f\in L^1(\G)_1$ such that the map
$$\LO\ni\omega\mapsto f\star\omega \in L^1(\G)$$
is isometric, then $L^1(\G)$ has the RNP.
\end{proposition}
\begin{proof} Assume that such $f\in\LO$ exists.
Let $\tilde f\in \Tr$ be a norm-preserving weak$^\ast$ extension of $f$. Then the map 
$$L^1(\G)\ni \omega\mapsto \tilde f\triangleleft\omega \in\T_\star(\G)$$
is an isometric embedding, where the action $\triangleleft$ is defined as \ref{mod}.
 To see this, let $\omega\in\LO$. Then we have 
\[
\|\om\| = \|f\star\om\| = \|\pi(\tilde f\triangleleft\om)\|\leq\|\tilde f\triangleleft\om\|
\leq\|\tilde f\|\,\|\om\| = \|\om\|. 
\]
 This implies that $L^1(\G)$ is isomorphic to a subspace of
$\Tr$, hence the claim follows from parts $(1)$ and $(2)$ of Proposition \ref{024}.
\end{proof}

Part $(3)$ of Proposition \ref{024}, at first glance, suggests that one might have
a dual version of this statement, saying that the Fourier algebra
$A(G)$ has the RNP if and only if $G$ is compact. But in fact, this is not the case.
A counter-example is given by the Fell group (see \cite[Remark 4.6]{Taylor}) which is non-compact, but its Fourier
algebra has the RNP.

Analogously to our earlier discussion on cohomological properties of $\G$, one may need to
take the operator space structure of  $\LO$ into account as well. 
Indeed, there is an operator space version of the RNP, due to Pisier (see \cite{Pisier3}),
which may be useful in this context.

But, in the following we point out another way of looking at this problem.
First, we give a general result.
\begin{theorem} \label{727727}
Let $M\subseteq \B(H)$ be a von Neumann algebra. Then the following are
equivalent:
\begin{itemize}
\item[(1)]
$M_*$ has the RNP;
\item[(2)]
there is a normal conditional expectation from $B(H)$ onto $M$.
\end{itemize}
\end{theorem}
\begin{proof} $(1)\Rightarrow (2):$ By \cite[Theorem 3.5]{Taylor}, $M$ is atomic,
 i.e., $M$ is an $l^\infty$-direct sum of $\B(H_i)$'s for
some Hilbert spaces $H_i$. So $M = N^{**}$ where $N=\oplus_\infty K(H_i)$ is
an ideal in $M$. Then, we have $(2)$ by \cite[Theorem 5]{Tomi3}.\\
$(2)\Rightarrow (1):$ The pre-adjoint map of the conditional
expectation defines an isometric embedding of $M_*$ into $\mathcal{T}(H)$.
Then in view of parts $(1)$ and $(2)$ of Proposition \ref{024}, we obtain $(2)$.
\end{proof}

\begin{corollary}\label{rnp11}
Let $\G$ be locally compact quantum group. Then the following are equivalent:
\begin{itemize}
 \item[(1)]
there exists a normal conditional expectation $E$ from $\B(\LT)$ onto $\LL$;
\item[(2)]
$\LO$ has the RNP.
\end{itemize}
\end{corollary}

In Theorem \ref{3310} we gave a characterization of discreteness of $\G$ in terms of existence
of a normal and covariant conditional expectation.
By comparing that result with Corollary \ref{rnp11} above, we see that the covariance
accounts precisely for the difference between the RNP and discreteness,
for $L^1$-algebras of locally compact quantum groups.

Next theorem shows that although discreteness and the
RNP are not equivalent in general for $\LO$, but 
with extra conditions on $\G$, that could be the case.
\begin{theorem}
Let $\G$ be a compact Kac algebra. Then the following are equivalent:
\begin{itemize}
\item[(1)] $\LO$ has the RNP;
\item[(2)] $\G$ is finite.
\end{itemize}
\end{theorem}
\begin{proof}

$(2)\Rightarrow (1)$ is obvious.
$(1)$ implies, by \cite[Theorem 3.5]{Taylor}, that $\LL$ is atomic, i.e.,
$\LL = \oplus_\infty \B(H_i)$ for some Hilbert spaces $H_i$. Since $\G$
is a compact Kac algebra, the Haar weight $\fee$ is a finite faithful trace, hence
all $H_i$'s are finite-dimensional and the restriction of $\fee$ to each $\B(H_i)$ is
its unique trace. Thus, $\fee = \sum_i tr_i$, and since $\fee$ is finite, there can be only
finitely many summands. So, $\G$ is finite.
%Theorem \ref{727727}, that there is
%a completely contractive map $\Phi:\LO\ra\Tr$ (the preadjoint of the conditional expectation),
%which is right inverse to the canonical
%quotient map $\pi:\Tr\ra\LO$. Hence, the short exact sequence
%\[
%0\ra I\ra\T_\star(\G)\xrightarrow{\pi} \LO
%\]
%defined in the proof of Theorem \ref{124312},
%is admmisible. Since $\G$ is a compact Kac algebra, $\LO$ is biprojective, with a completely
%contractive map $\psi:\LO\ra\LO\tp\LO$ as a right inverse to the product
%$\mathfrak{m}:\LO\tp\LO\ra\LO$, by \cite[Theorem 4.2]{Daws}.
%Therefore, $\T_\star(\G)$ is biprojective by \cite[Lemma 4.2]{Pir}.
%Hence, $\G$ is finite by Theorem \ref{124312}.
\end{proof}
As a special case we obtain the following which can also be deduced from a result
by Lau--\"Ulger \cite[Theorem 4.3]{LauUlg} stating that for an [IN] locally compact group $G$,
the Fourier algebra $A(G)$ has the RNP if and only if $G$ is compact.
\begin{corollary}
Let $G$ be a discrete group. Then the following are equivalent:
\begin{itemize}
\item[(1)] the Fourier algebra $A(G)$ has the RNP;
\item[(2)] $G$ is finite.
\end{itemize}
\end{corollary}

\bibliographystyle{amsplain}

\end{document}